\documentclass[preprint,11pt]{elsarticle}
\usepackage{mathrsfs}
\usepackage{amsfonts}
\usepackage{amsmath}
\usepackage{amstext}
\usepackage{stmaryrd}
\usepackage{amssymb}
\usepackage{amsthm}
\usepackage{mathrsfs}
\usepackage{amsfonts}
\usepackage{enumerate}
\usepackage{amscd}
\usepackage{indentfirst}
\usepackage{enumerate}
\usepackage{amsmath,amsfonts,amssymb,amsthm}
\usepackage{amsmath,amssymb,amsthm,amscd}
\usepackage{graphicx,mathrsfs}
\usepackage{appendix}

\numberwithin{equation}{section}

\newtheorem{theorem}{Theorem}[section]

\newtheorem{lemma}[theorem]{Lemma}

\theoremstyle{definition}

\newtheorem{definition}[theorem]{Definition}

\theoremstyle{remark}
\newtheorem{remark}[theorem]{Remark}

\begin{document}
\begin{frontmatter}

\title{Existence of Steady Multiple Vortex Patches to the Vortex-wave System
}

\author{Daomin Cao}
\ead{dmcao@amt.ac.cn}
\author{Guodong Wang}

\ead{wangguodong14@mails.ucas.ac.cn}

\begin{abstract}
In this paper we prove the existence of steady multiple vortex patch solutions to the vortex-wave system in a planar bounded domain. The construction is performed by solving a certain variational problem for the vorticity and studying its asymptotic behavior as the vorticity strength goes to infinity.
\end{abstract}
\end{frontmatter}


\section{Introduction}
The evolution of an incompressible, inviscid fluid is described by the following Euler equation
\begin{equation}\label{01}
\begin{cases}
\partial_t\mathbf{v}+(\mathbf{v}\cdot\nabla)\mathbf{v}=-\nabla P,\\
\nabla\cdot \mathbf{v}=0,
\end{cases}
\end{equation}
where $\mathbf{v}$ is the velocity field and $P$ is the pressure. In the planar case $\mathbf{v}=(\mathbf{v}_1,\mathbf{v}_2)$ and we introduce the vorticity of the fluid \[\omega:=\partial_1\mathbf{v}_2-\partial_2\mathbf{v}_1,\] then we obtain the following vorticity formulation of the Euler equation
\begin{equation}\label{02}
\partial_t\omega+\mathbf{v}\cdot\nabla\omega=0.
\end{equation}
In the whole plane the velocity can be recovered from the vorticity through the Biot-Savart law, that is,
\begin{equation}
\mathbf{v}(x,t)=K*\omega(x,t):=-\frac{1}{2\pi}\int_{\mathbb{R}^2}\frac{J(x-y)}{|x-y|^2}\omega(y,t)dy,
\end{equation}
where $J(a,b):=(b,-a)$ denotes clockwise rotation through $\frac{\pi}{2}$ for any vector $(a,b)\in \mathbb{R}^2$, and $K(x)=-\frac{1}{2\pi}\frac{Jx}{|x|^2}$ is called the Biot-Savart kernel. The vorticity equation \eqref{02} means that the vorticity $\omega$ is transported by the divergence-free velocity field $K*\omega$.

In some cases the vorticity is sharply concentrated in $N$ small disjoint regions, then its time evolution is approximately described by the point vortex model(also called the Kirchhoff-Routh equation, see \cite{L} for example)
\begin{equation}\label{03}
\frac{dx_i}{dt}=\sum_{j=1,j\neq i}^N\kappa_jK(x_i-x_j),\quad\;i=1,\cdot\cdot\cdot,N,
\end{equation}
where $x_i$ is the position of the $i-$th point vortex and $\kappa_i$ is the corresponding vorticity strength. According to the point vortex model, for each point vortex concentrated at $x_i$, the velocity it produces is $\kappa_iK(\cdot-x_i)$, and it moves by the velocity field produced by all the other $N-1$ point vortices(not included itself). The asymptotic nature of the point vortex approximation has been analyzed in many papers, see \cite{CPY,MPa,SV,T3} for example.

Now it is natural to consider the mixed problem in which the vorticity consists of a continuously distributed part and $N$ concentrated vortices. Marchioro and Pulvirenti firstly considered this problem in \cite{MPu2} and they called it the vortex-wave system. In the whole plane the vortex-wave system is written as follows
\begin{equation}\label{04}
\begin{cases}
 \partial_t{\omega}+\mathbf{v}\cdot\nabla\omega=0,
 \\ \frac{dx_i}{dt}=K*\omega(x_i,t)+\sum_{j=1,j\neq i}^N\kappa_jK(x_i-x_j),\,\,i=1,\cdot\cdot\cdot,N,
  \\ \mathbf{v}=K*\omega+\sum_{j=1}^k\kappa_jK(\cdot-x_j).
\end{cases}
\end{equation}
Throughout this paper we call $\omega$ the background vorticity. System \eqref{04} means that the background vorticity is transported by the velocity field produced by itself(the term $K*\omega$) and the $N$ point vortices(the term $\sum_{j=1}^k\kappa_jK(\cdot-x_j)$), and each point vortex moves by the velocity produced by the background vorticity(the term $K*\omega(x_i,t)$) and all the other $N-1$ point vortices(the term $\sum_{j\neq i,j=1}^N\kappa_jK(\cdot-x_j)$). By constructing Lagrangian paths Marchioro and Pulvirenti in \cite{MPu2} proved an existence theorem for initial background vorticity belonging to $L^1(\mathbb{R}^2)\cap L^\infty(\mathbb{R}^2)$. More existence and uniqueness results can be found in \cite{B,LM,LM2,Mi}.

In this paper we focus on the steady vortex-wave system in a bounded domain. In this case the Kirchhoff-Routh function plays an essential role. On the one hand, any critical point of the Kirchhoff-Routh function is a stationary point of the vortex model, and it has been shown in \cite{CPY} that if this critical point is non-degenerate, then there exists a family of steady vortex patch solutions of the Euler equation shrinking to that critical point. Similar results can also be found in \cite{CLW,SV,T,W}. On the other hand, if a family of vortex patch solutions of the Euler equation shrinks to a point, then this point must be a critical point of the Kirchhoff-Routh function, see \cite{CGPY}. In this way, the Kirchhoff-Routh function establishes connection between the Euler equation and the vortex model.

Our purpose in the present paper is to extend the result in \cite{CPY} to the vortex-wave system.
To be more precise, we prove that for any given strict local minimum point of the Kirchhoff-Routh function, there exists a family of steady vortex patch solutions to the vortex-wave system that shrinks to this point. Here by vortex patch solution we mean that the background vorticity $\omega$ is a constant on some planar set while $\omega=0$ elsewhere.

The method we use in this paper to construct steady solutions is called the vorticity method, which is first established by Arnold and further developed by many authors, see \cite{A,B2,B4,BM,EM,EM2,T,T2} for example. Roughly speaking, the vorticity method is to maximize the kinetic energy of the fluid under some constraints for the vorticity. For the Euler equation, the kinetic energy of the fluid with bounded vorticity is always finite; but for the vortex-wave system, the kinetic energy is infinite due to the existence of point vortices. To overcome this difficulty, we drop the infinite self-energy term for each point vortex. We refer the interested reader to \cite{CW2} where the energy is calculated rigorously for the vortex-wave system with a single point vortex.

It is worth mentioning that the construction in \cite{CPY} is based on a finite dimensional reduction argument. The advantage of the method in \cite{CPY} is that solutions concentrating at a given saddle point of the Kirchhoff-Routh function can be constructed.
However, non-degenerate condition of this saddle point is required in this situation. Using the vorticity method, we are able to construct solutions concentrating at a given strict local minimum point of the Kirchhoff-Routh function, no matter whether this point is degenerate or not. Another advantage of the vorticity method is that we can analyze the energy of the solution, which is helpful to prove stability, see \cite{B3,CW} for example.

This paper is organized as follows. In Section 2 we first introduce some notations and the vortex-wave system in bounded domains, then we state the main result. In Section 3, we solve a certain variational problem for the vorticity and study its asymptotic behavior. In Section 4 we prove the main result.

\section{Main result}
\subsection{Notations}

Let $D\subset\mathbb{R}^2$ be a bounded and simply connected domain with smooth boundary. The Green's function for $-\Delta$ in $D$ with zero Dirichlet data on $\partial D$ is written as follows:
\begin{equation}\label{21}
G(x,y)=\frac{1}{2\pi}\ln \frac1{|x-y|}-h(x,y), \,\,\,x,y\in D.
\end{equation}
$H(x):=\frac{1}{2}h(x,x)$ is called the Robin function. Let $k$ be a positive integer, $\kappa_i\in \mathbb{R}\setminus{\{0\}}$, $i=1,\cdot\cdot\cdot,k$, then the Kirchhoff-Routh function of $D$ is defined by
\begin{equation}\label{2-2}
\mathcal{K}_k(x_1,\cdot\cdot\cdot,x_k)=-\sum_{i\neq j, 1\leq i,j\leq k}\kappa_i\kappa_jG(x_i,x_j)+\sum_{i=1}^k\kappa_i^2h(x_i,x_i),
\end{equation}
where $x_i\in D$ and $x_i\neq x_j$ for $i\neq j$. Note that if $k=1$, then $\mathcal{K}_1=2{\kappa^2}H.$

Throughout this paper we will use the following notations. $supp(g)$ denotes the support of some function $g$, and $sgn(\cdot)$ denotes the sign function, that is, for any $a\in\mathbb{R}$,
\begin{equation}
sgn(a):=
\begin{cases}
1,\quad\,\,\,\,\,\text{if}\;a>0,\\
0,\quad\,\,\,\,\,\text{if}\;a=0,\\
-1,\,\,\,\,\,\,\text{if}\;a<0.
\end{cases}
\end{equation}
For any Lebesgue measurable set $A\subset \mathbb{R}^2$, $|A|$ denotes the two-dimensional Lebesgue measure of $A$, $I_A$ denotes the characteristic function of $A$, that is, $I_A(x)=1$ if $x\in A$ and $I_A(x)=0$ elsewhere, $\overline{A}$ denotes the closure of $A$ in the Euclidean topology,
and $diam(A)$ denotes the diameter of $A$, that is,
\begin{equation}
diam(A)=\sup_{x,y\in A}|x-y|.
\end{equation}

\subsection{Vortex-wave system in bounded domains}
Now we consider the incompressible steady flow in $D$ with impermeability boundary condition. The evolution of the velocity field $\mathbf{v}=(\mathbf{v}_1,\mathbf{v}_2)$ and the pressure $P$ confined in $D$ is described by the following Euler equation
\begin{equation}\label{22}
\begin{cases}
\partial_t\mathbf{v}+\mathbf{v}\cdot\nabla\mathbf{v}=-\nabla P,\quad\,\, \text{in}\; D\times(0,+\infty),\\
\nabla\cdot \mathbf{v}=0,\quad\quad\quad \quad\quad\quad\,\,\,\text{in}\; D\times(0,+\infty),\\
\mathbf{v}\cdot\nu=0,\quad\quad\quad \quad\quad\quad\,\,\,\,\text{on}\; \partial D\times(0,+\infty),
\end{cases}
\end{equation}
 where $\nu$ is the outward unit normal of $\partial D$.

Taking the curl on both sides of the first equation of \eqref{22}, we get the vorticity equation
 \begin{equation}\label{23}
\partial_t\omega+\mathbf{v}\cdot\nabla\omega=0.
 \end{equation}
Since $\mathbf{v}$ is divergence-free, there is a function $\psi$, which is called the stream function, such that
\begin{equation}
\mathbf{v}=J\nabla\psi=(\partial_2\psi,-\partial_1\psi).
\end{equation}
By the definition of $\omega$, it is easy to see that
\begin{equation}\label{24}
-\Delta\psi=\omega.
\end{equation}

The impermeability boundary condition given in the third equation of \eqref{22} implies that $\psi$ is a constant on each connected component of $\partial D$.  Since $D$ is simply connected, after suitably adding a constant to $\psi$ we can assume
\begin{equation}\label{25}
\psi(x,t)=0,\quad x\in\partial D.
\end{equation}
By \eqref{24} and \eqref{25}, we have
\begin{equation}
\psi(x,t)=(-\Delta)^{-1}\omega(x,t):=\int_DG(x,y)\omega(y,t)dy,\quad\;x\in D.
\end{equation}

Now we introduce the notation \[\partial(f,g):=\nabla f\cdot J\nabla g=\partial_1f\partial_2g-\partial_2f\partial_1g,\]
then the vorticity equation \eqref{23} can be written as
\begin{equation}
\begin{cases}
\partial_t\omega+\partial(\omega,\psi)=0,\\
\psi=(-\Delta)^{-1}\omega.
\end{cases}
\end{equation}

If the vorticity is a Dirac delta measure(or a point vortex) located at $x\in D$, i.e., $\omega=\delta(x)$, then formally the velocity field it produces is
\[J\nabla(-\Delta)^{-1}\delta(x)=J\nabla G(x,\cdot)=\frac{1}{2\pi}\frac{J(x-\cdot)}{|x-\cdot|^2}-J\nabla h(x,\cdot).\]
This velocity field is singular at $x$. If we drop the term $\frac{1}{2\pi}\frac{J(x-\cdot)}{|x-\cdot|^2}$, that is, we assume that the velocity at $x$ is $-J\nabla h(x,\cdot)\big|_x=-J\nabla H(x)$, then the evolution of this point vortex is described by the following ordinary differential equation:
\begin{equation}
\frac{dx}{dt}=-J\nabla H(x).
\end{equation}
Similarly, the evolution of $l$ point vortices is described by the following system of equations:
\begin{equation} \label{26}
\frac{dx_i}{dt}=-\kappa_iJ\nabla H(x_i)+\sum_{j=1,j\neq i}^l\kappa_jJ\nabla_{x_i} G(x_j,x_i),\quad\;i=1,\cdot\cdot\cdot,l,
\end{equation}
where $\kappa_i$ is the vorticity strength of the $i$-th point vortex. Equation \eqref{26} is called the Kirchhoff-Routh equation. It is easy to see that the Kirchhoff-Routh function $\mathcal{K}_l$ is exactly the Hamiltonian of the Kirchhoff-Routh equation.

Now we consider the mixed problem, that is, the vorticity consists of a continuously distributed part $\omega$ and $l$ point vortices $x_i$ with strength $\kappa_i$, $i=1,\cdot\cdot\cdot,l$, then the evolution of $\omega$ and $x_i$ will obey the following equations
\begin{equation}\label{27}
\begin{cases}
\partial_t\omega+\partial(\omega,\psi+\sum_{j=1}^l\kappa_jG(x_j,\cdot))=0\\
\frac{dx_i}{dt}=J\nabla(\psi+\sum_{j=1,j\neq i}^l\kappa_jG(x_j,\cdot)-\kappa_i H)(x_i),\,\,i=1,\cdot\cdot\cdot,l,
\end{cases}
\end{equation}
where $\psi=(-\Delta)^{-1}\omega$. \eqref{27} is called the vortex-wave system in $D$.

In this paper we restrict ourselves to the stationary case, that is, we consider the following equations
\begin{equation}\label{28}
\begin{cases}
\partial(\omega,\psi+\sum_{j=1}^l\kappa_jG(x_j,\cdot))=0,\\
\nabla(\psi+\sum_{j=1,j\neq i}^l\kappa_jG(x_j,\cdot)-\kappa_i H)(x_i)=0,\,\,i=1,\cdot\cdot\cdot,l,
\end{cases}
\end{equation}
where $\psi=(-\Delta)^{-1}\omega$.

Since we are going to deal with vortex patch solutions which are discontinuous, it is necessary to give the weak formulation of \eqref{28}.

\begin{definition}\label{2-10}
Let $\omega\in L^\infty(D)$, $x_i\in D, i=1,\cdot\cdot\cdot,l,$ then $(\omega,x_1,\cdot\cdot\cdot,x_l)$ is called a weak solution to \eqref{28} if it satisfies
\begin{equation}\label{29}
\begin{cases}
\int_\Omega\omega(x)\partial(\psi(x)+\sum_{j=1}^l\kappa_jG(x_j,x),\phi(x))dx=0,\,\,\forall \phi\in C_c^\infty(\Omega),\\
\nabla(\psi(x)+\sum_{j=1,j\neq i}^l\kappa_jG(x_j,x)-\kappa_i H(x))\big|_{x=x_i}=0,\,\,i=1,\cdot\cdot\cdot,l,
\end{cases}
\end{equation}
where $\psi=(-\Delta)^{-1}\omega$.
\end{definition}
\begin{remark}
Note that since $\omega\in L^\infty(D)$, by $L^p$ estimate $\psi\in W^{2,p}(D)$ for any $1<p<+\infty$, then by Sobolev embedding $\psi\in C^{1,\alpha}(\overline{D})$ for any $0<\alpha<1.$
\end{remark}
\begin{remark}
Definition \ref{2-10} can be derived formally from \eqref{28} by integration by parts, see \cite{CW2} for the detailed calculation.
\end{remark}

\subsection{Main result}

Our main result in this paper is the following theorem:
\begin{theorem}\label{230}
Let $k,p,l$ be positive integers such that $p+l=k$, and $\kappa_i, i=1,\cdot\cdot\cdot,k,$ be $k$ real numbers such that $\kappa_i\neq 0$. Suppose that $(\bar{x}_1,\cdot\cdot\cdot,\bar{x}_k)$ is a strict local minimum point of $\mathcal{K}_k$ defined by \eqref{2-2}, where $\bar{x}_i\in D$ and $\bar{x}_i\neq\bar{x}_j$ for $i\neq j$. Then there exists $\lambda_0>0$ such that for $\lambda>\lambda_0$, \eqref{29} has a solution $(\omega^\lambda,x_{p+1}^\lambda,\cdot\cdot\cdot,x_{k}^\lambda)$ satisfying
\begin{equation}\label{231}
\omega^\lambda=\sum_{i=1}^p\omega_i^\lambda,\quad \int_D\omega^\lambda_i=\kappa_i, \quad\omega^\lambda_i=sgn(\kappa_i)\lambda I_{A_i^\lambda},\quad i=1,\cdot\cdot\cdot,p,
\end{equation}
where $A^\lambda_i$ has the form
\begin{equation}\label{232}
A_i^\lambda=\{x\in D|sgn(\kappa_i)(\psi^\lambda(x)+\sum_{j=p+1}^k\kappa_jG(x_j^\lambda,x))>c^\lambda_i\}\cap B_{\delta}(\bar{x}_i)
\end{equation}
for some $c^\lambda_i>0$ and $\delta>0$($\delta$ is independent of $\lambda$). Moreover,
\begin{equation}\label{233}
diam(A^\lambda_i)\leq C\lambda^{-\frac{1}{2}},\; \lim_{\lambda\rightarrow+\infty}|\int_Dx\omega^\lambda_i(x)dx-\bar{x}^\lambda_i|=0, \quad i=1,\cdot\cdot\cdot,p,
\end{equation}
\begin{equation}\label{234}
\lim_{\lambda\rightarrow+\infty}|x^\lambda_j-\bar{x}_j|=0, \quad j=p+1,\cdot\cdot\cdot,k,
\end{equation}
where $C$ is a positive number not depending on $\lambda$.
\end{theorem}

\begin{remark}
By \eqref{233}, we know that $A^\lambda_i$ shrinks to $\bar{x}^\lambda_i$ as $\lambda\rightarrow+\infty$, that is,
 \[\lim_{\lambda\rightarrow+\infty}\sup_{x\in A^\lambda_i}|x-\bar{x}^\lambda_i|=0,\quad i=1,\cdot\cdot\cdot,p.\]
 It is also easy to see that $\overline{A^\lambda_i}\subset B_\delta(\bar{x}_i)$ for sufficiently large $\lambda$.
\end{remark}

\section{Variational problem}

Let $(\bar{x}_1,\cdot\cdot\cdot,\bar{x}_k)$ be a strict local minimum point of $\mathcal{K}_k$, where $\bar{x}_i\in D$ and $\bar{x}_i\neq\bar{x}_j$ for $i\neq j$. Without loss of generality, we assume that $(\bar{x}_1,\cdot\cdot\cdot,\bar{x}_k)$ is the unique minimum point of $\mathcal{K}_k$ on $\overline{B_{\delta_0}(\bar{x}_1)}\times\cdot\cdot\cdot\times\overline{B_{\delta_0}(\bar{x}_k)}$, where $\delta_0>0$ is a small positive number such that $\overline{B_{\delta_0}(\bar{x}_i)}\subset D$ and $\overline{B_{\delta_0}(\bar{x}_i)}\cap\overline{B_{\delta_0}(\bar{x}_j)}=\varnothing$ for all $i,j=1,\cdot\cdot\cdot,k$ and $i\neq j$.

\begin{remark}
 We know no general result that guarantees the existence of strict local minimum point of $\mathcal{K}_k$ for $k\geq 2$. Some special cases are as follows: if $k=1$ and $D$ is convex, by \cite{CF1} $\mathcal{K}_1$ is a strictly convex function in $D$, thus has a unique minimum point; if $k\geq 2, \kappa_i>0$ for each $i$ and $D$ is convex, by \cite{GT} there does not exist any critical point of $\mathcal{K}_k$; if $k=2$, some examples of strict local minimum point of $\mathcal{K}_2$ are given computationally in \cite{EM2}. More related results can also be found in \cite{BP,BPW} and the reference therein.
\end{remark}

Let $\lambda$ be a positive real number. Define
\begin{equation}
\begin{split}
N_p^\lambda=&\{\omega\in L^\infty(D)|\;\omega=\sum_{i=1}^p\omega_i, supp(\omega_i)\subset B_\delta(\bar{x}_i), \int_D\omega_i=\kappa_i, 0\leq sgn(\kappa_i)\omega_i\leq \lambda,\, i=1,\cdot\cdot\cdot,p\},
\end{split}
\end{equation}
where $\delta<\frac{\delta_0}{2}$ is a small positive number to be determined later.

For $(\omega,x_{p+1},\cdot\cdot\cdot,x_k)\in N_p^\lambda\times \overline{B_\delta(\bar{x}_{p+1})}\times\cdot\cdot\cdot\times\overline{B_\delta(\bar{x}_k)}$, define
\begin{equation}
\begin{split}
\mathcal{E}(\omega,x_{p+1},\cdot\cdot\cdot,x_k):=&E(\omega)+\sum_{j=p+1}^k\kappa_j\psi(x_j)
+\frac{1}{2}\sum_{\begin{subarray}\,\,\,\quad i\neq j\\ p+1\leq i,j\leq k \end{subarray}}\kappa_i\kappa_j G(x_i,x_j)\\
&-\sum_{j=p+1}^k\kappa_j^2H(x_j),
\end{split}
\end{equation}
where $E(\omega):=\frac{1}{2}\int_D\int_DG(x,y)\omega(x)\omega(y)dxdy$ and $\psi(x)=(-\Delta)^{-1}\omega(x):=\int_DG(x,y)\omega(y)dy$.

Let us explain the definition of $\mathcal{E}$ briefly: the first term $E(\omega)$ represents the self-interacting energy of the background vorticity $\omega$, the second term represents the mutual interaction energy between the background vorticity and the $l$ point vortices, the third term represents the total interaction energy between any two different point vortices, and the fourth term represents the interaction energy between the $l$ point vortices and the boundary of $D$. As we have mentioned in Section 1, we have dropped the infinite self-interacting energy for each point vortex.

Now we consider the maximization of $\mathcal{E}$ on $N_p^\lambda\times \overline{B_\delta(\bar{x}_{p+1})}\times\cdot\cdot\cdot\times\overline{B_\delta(\bar{x}_k)}$.

\begin{lemma}
There exists $(\omega^\lambda, x^\lambda_{p+1},\cdot\cdot\cdot,x^\lambda_k)\in N_p^\lambda\times \overline{B_\delta(\bar{x}_{p+1})}\times\cdot\cdot\cdot\times\overline{B_\delta(\bar{x}_k)}$, such that
\begin{equation}
\mathcal{E}(\omega^\lambda,x^\lambda_{p+1},\cdot\cdot\cdot,x^\lambda_k)=\sup_{(\omega, x_{p+1},\cdot\cdot\cdot,x_k)\in N_p^\lambda\times \overline{B_\delta(\bar{x}_{p+1})}\times\cdot\cdot\cdot\times\overline{B_\delta(\bar{x}_k)}}\mathcal{E}(\omega,x_{p+1},\cdot\cdot\cdot,x_k).
\end{equation}
\end{lemma}
\begin{proof}
Firstly for any $(\omega, x_{p+1},\cdot\cdot\cdot,x_k)\in N_p^\lambda\times \overline{B_\delta(\bar{x}_{p+1})}\times\cdot\cdot\cdot\times\overline{B_\delta(\bar{x}_k)}$, since $G\in L^1(D\times D)$, $dist(B_{\delta_0}(\bar{x}_i),B_{\delta_0}(\bar{x}_j))>0$ for any $i\neq j, p+1\leq i,j\leq k$, and $dist(B_{\delta_0}(\bar{x}_i),\partial D)>0, p+1\leq i\leq k$, we have
\begin{equation}
\begin{split}
\mathcal{E}(\omega,x_{p+1},\cdot\cdot\cdot,x_k)&=\frac{1}{2}\int_D\int_DG(x,y)\omega(x)\omega(y)dxdy+\sum_{j=p+1}^k\kappa_j\psi(x_j)\\
&+\frac{1}{2}\sum_{\begin{subarray}\,\,\,\quad i\neq j\\ p+1\leq i,j\leq k \end{subarray}}\kappa_i\kappa_j G(x_i,x_j)-\sum_{j=p+1}^k\kappa_j^2H(x_j)\\
&\leq\frac{\lambda^2}{2}\int_D\int_D|G(x,y)|dxdy+\sum_{j=p+1}^k|\kappa_j||\psi|_{L^\infty(D)}+C\\
&\leq C,
\end{split}
\end{equation}
where $\psi=(-\Delta)^{-1}\omega$, $C$ is a positive constant depending on $\lambda, \delta_0, D$ and $(\bar{x}_{p+1},\cdot\cdot\cdot,\bar{x}_k).$  So we have
\[M:=\sup_{(\omega, x_{p+1},\cdot\cdot\cdot,x_k)\in N_p^\lambda\times \overline{B_\delta(\bar{x}_{p+1})}\times\cdot\cdot\cdot\times\overline{B_\delta(\bar{x}_k)}}\mathcal{E}(\omega,x_{p+1},\cdot\cdot\cdot,x_k)<+\infty.\]

Now we can choose $(\omega^n, x^n_{p+1},\cdot\cdot\cdot,x^n_k)\in N_p^\lambda\times \overline{B_\delta(\bar{x}_{p+1})}\times\cdot\cdot\cdot\times\overline{B_\delta(\bar{x}_k)}$ such that
\[\lim_{n\rightarrow+\infty}\mathcal{E}(\omega^n, x^n_{p+1},\cdot\cdot\cdot,x^n_k)=M.\]
Since $N_p^\lambda$ is weakly star closed in $L^\infty(D)$(see \cite{CW3}, Theorem 2.1) and $\overline{B_\delta(\bar{x}_{j})}$ is closed in the Euclidean topology for $j=p+1,\cdot\cdot\cdot,k,$ there exists $(\omega^\lambda, x^\lambda_{p+1},\cdot\cdot\cdot,x^\lambda_k)\in N_p^\lambda\times \overline{B_\delta(\bar{x}_{p+1})}\times\cdot\cdot\cdot\times\overline{B_\delta(\bar{x}_k)}$ such that(up to a subsequence)
\[\omega^n\rightarrow\omega^\lambda ,\quad\text{weakly star in $L^\infty(D)$},\]
\[x^n_j\rightarrow x^\lambda_j,\quad j=p+1,\cdot\cdot\cdot,k.\]
Then obviously
\[\mathcal{E}(\omega^\lambda, x^\lambda_{p+1},\cdot\cdot\cdot,x^\lambda_k)=\lim_{n\rightarrow+\infty}\mathcal{E}(\omega^n, x^n_{p+1},\cdot\cdot\cdot,x^n_k)=M,\]
which completes the proof.
\end{proof}

\begin{remark}
It is easy to see that
\begin{equation}\label{35}
E(\omega)+\sum_{j=p+1}^k\kappa_j\psi(x_j^\lambda)\leq E(\omega^\lambda)+\sum_{j=p+1}^k\kappa_j\psi^\lambda(x_j^\lambda),\quad \forall\; \omega\in N_p^\lambda,
\end{equation}
and
\begin{equation}\label{36}
\begin{split}
&\sum_{j=p+1}^k\kappa_j\psi^\lambda(x_j)
+\frac{1}{2}\sum_{\begin{subarray}\,\,\,\quad i\neq j\\ p+1\leq i,j\leq k \end{subarray}}\kappa_i\kappa_j G(x_i,x_j)-\sum_{j=p+1}^k\kappa_j^2H(x_j)\\
\leq &\sum_{j=p+1}^k\kappa_j\psi^\lambda(x_j^\lambda)
+\frac{1}{2}\sum_{\begin{subarray}\,\,\,\quad i\neq j\\ p+1\leq i,j\leq k \end{subarray}}\kappa_i\kappa_j G(x^\lambda_i,x^\lambda_j)-\sum_{j=p+1}^k\kappa_j^2H(x^\lambda_j),\\
&\forall \;(x_{p+1},\cdot\cdot\cdot,x_k)\in \overline{B_\delta(\bar{x}_{p+1})}\times\cdot\cdot\cdot\times\overline{B_\delta(\bar{x}_k)},
\end{split}
\end{equation}
where $\psi^\lambda=(-\Delta)^{-1}\omega^\lambda$ and $\psi=(-\Delta)^{-1}\omega$.
\end{remark}

Since $\omega^\lambda\in N_p^\lambda$, we can write $\omega^\lambda=\sum_{i=1}^p\omega^\lambda_i$, where $\int_D\omega^\lambda_i=\kappa_i, supp(\omega^\lambda_i)\subset {B_\delta(\bar{x}_i)}$ and $0\leq sgn(\kappa_i)\omega^\lambda_i\leq \lambda.$
The profile of each $\omega^\lambda_i$ is as follows.
\begin{lemma}\label{3-007}
For $i=1,\cdot\cdot\cdot,p$, $\omega^\lambda_i$ has the form
\[\omega^\lambda_i=sgn(\kappa_i)\lambda I_{A^\lambda_i},\]
 where
\[A^\lambda_i:={\{x\in D| sgn(\kappa_i)(\psi^\lambda(x)+\sum_{j=p+1}^k\kappa_jG(x_j^\lambda,x))>c^\lambda_i\}\cap B_\delta(\bar{x}_i)}\]
 for some $c^\lambda_i\in \mathbb{R}$.
\end{lemma}

\begin{proof}
To make it clear we divide the proof into two cases.

Case 1: $\kappa_i>0$. For  $s>0$  we define a family of test functions $\omega^\lambda_{s}=\omega^\lambda+s(z_0-z_1)$, where
\begin{equation}
\begin{cases}
z_0,z_1\in L^\infty(D), z_0,z_1\geq 0, \int_Dz_0(x)dx=\int_D z_1(x)dx,
 \\ supp(z_0),supp(z_1)\subset B_\delta(\bar{x}_i),
 \\z_0=0\text{\,\,\,\,\,\,} \text{in}\text{\,\,} B_\delta(\bar{x}_i)\setminus\{{\omega}^\lambda_i\leq \lambda-\mu\},
 \\z_1=0\text{\,\,\,\,\,\,} \text{in}\text{\,\,} B_\delta(\bar{x}_i)\setminus\{{\omega}^\lambda_i\geq\mu\},
\end{cases}
\end{equation}
for some $0<\mu<\lambda$. It is not hard to check that for fixed $z_0,z_1$ and $\mu$, if $s$ is sufficiently small(depending on  $z_0,z_1,\mu$), then $\omega^\lambda_s\in N_p^\lambda$. By \eqref{35} we have
\[\frac{d}{ds}[E(\omega^\lambda_s)+\sum_{j=p+1}^k\kappa_j\psi^\lambda_s(x^\lambda_j)]\bigg{|}_{s=0^+}\leq 0,\]
where $\psi^\lambda_s=(-\Delta)^{-1}\omega^\lambda_s$.
That is,
\[\int_D(\psi^\lambda(x)+\sum_{j=p+1}^k\kappa_jG(x^\lambda_j,x))z_0(x)dx\leq\int_D(\psi^\lambda(x)+\sum_{j=p+1}^k\kappa_jG(x^\lambda_j,x))z_1(x)dx.\]
Since $z_0,z_1,\mu$ are chosen arbitrarily as above we have
\[\sup_{\{\omega^\lambda_i<\lambda\}\cap B_\delta(\bar{x}_i)}(\psi^\lambda+\sum_{j=p+1}^k\kappa_jG(x^\lambda_j,\cdot))\leq
\inf_{\{\omega^\lambda_i>0\}\cap B_\delta(\bar{x}_i)}(\psi^\lambda+\sum_{j=p+1}^k\kappa_jG(x^\lambda_j,\cdot)).\]
But $\psi^\lambda+\sum_{j=p+1}^k\kappa_jG(x^\lambda_j,\cdot)$ is a continuous function on $\overline{B_\delta(\bar{x}_i)}$(notice that $x^\lambda_j\notin \overline{B_\delta(\bar{x}_i)}$ for $p+1\leq j\leq k$), we obtain
\[\sup_{\{\omega^\lambda_i<\lambda\}\cap B_\delta(\bar{x}_i)}(\psi^\lambda+\sum_{j=p+1}^k\kappa_jG(x^\lambda_j,\cdot))=
\inf_{\{\omega^\lambda_i>0\}\cap B_\delta(\bar{x}_i)}(\psi^\lambda+\sum_{j=p+1}^k\kappa_jG(x^\lambda_j,\cdot)).\]
Now we define
\begin{equation}\label{38}
\begin{split}
c^\lambda_i&:=\sup_{\{\omega^\lambda_i<\lambda\}\cap B_\delta(\bar{x}_i)}(\psi^\lambda+\sum_{j=p+1}^k\kappa_jG(x^\lambda_j,\cdot))\\
&=
\inf_{\{\omega^\lambda_i>0\}\cap B_\delta(\bar{x}_i)}(\psi^\lambda+\sum_{j=p+1}^k\kappa_jG(x^\lambda_j,\cdot)).
\end{split}
\end{equation}
It is easy to see that
\begin{equation}\label{39}
\omega^\lambda_i\equiv \lambda\quad \text{a.e. on} \;\{x\in D|\psi^\lambda(x)+\sum_{j=p+1}^k\kappa_jG(x^\lambda_j,x)>c^\lambda_i\}\cap B_\delta(\bar{x}_i),
\end{equation}
\begin{equation}\label{310}
\omega^\lambda_i\equiv 0\quad \text{a.e. on} \;\{x\in D|\psi^\lambda(x)+\sum_{j=p+1}^k\kappa_jG(x^\lambda_j,x)<c^\lambda_i\}\cap B_\delta(\bar{x}_i).
\end{equation}
On the set $\{x\in D|\psi^\lambda(x)+\sum_{j=p+1}^k\kappa_jG(x^\lambda_j,x)=c^\lambda_i\}\cap B_\delta(\bar{x}_i)$, we have \begin{equation}
\nabla(\psi^\lambda+\sum_{j=p+1}^k\kappa_jG(x^\lambda_j,\cdot))=0,\quad\text{a.e.},
\end{equation}
which gives
\begin{equation}\label{312}
\omega^\lambda_i=-\Delta(\psi^\lambda+\sum_{j=p+1}^k\kappa_jG(x^\lambda_j,\cdot))=0,\quad\text{a.e.}.
\end{equation}
\eqref{39},\eqref{310} and \eqref{312} together give
\begin{equation}
\omega^\lambda_i=\lambda I_{\{\psi^\lambda+\sum_{j=p+1}^k\kappa_jG(x_j^\lambda,\cdot)>c^\lambda_i\}\cap B_\delta(\bar{x}_i)},
\end{equation}
which completes the proof of Case 1.

Case 2: $\kappa_i<0$. The argument is similar. For $s>0$, define $\omega^\lambda_{s}=\omega^\lambda+s(z_1-z_0)$, where
\begin{equation}
\begin{cases}
z_0,z_1\in L^\infty(D), z_0,z_1\geq 0, \int_Dz_0(x)dx=\int_D z_1(x)dx,
 \\ supp(z_0),supp(z_1)\subset B_\delta(\bar{x}_i),
 \\z_0=0\text{\,\,\,\,\,\,} \text{in}\text{\,\,} B_\delta(\bar{x}_i)\setminus\{{\omega}^\lambda_i\geq\mu-\lambda\},
 \\z_1=0\text{\,\,\,\,\,\,} \text{in}\text{\,\,} B_\delta(\bar{x}_i)\setminus\{{\omega}^\lambda_i\leq-\mu\},
\end{cases}
\end{equation}
for some $0<\mu<\lambda$. Then
\[\frac{d}{ds}[E(\omega^\lambda_s)+\sum_{j=p+1}^k\kappa_j\psi^\lambda_s(x^\lambda_j)]\bigg{|}_{s=0^+}\leq 0.\]
Repeat the argument in Case 1 we obtain
\[\sup_{\{\omega^\lambda_i<0\}\cap B_\delta(\bar{x}_i)}(\psi^\lambda+\sum_{j=p+1}^k\kappa_jG(x^\lambda_j,\cdot))=
\inf_{\{\omega^\lambda_i>-\lambda\}\cap B_\delta(\bar{x}_i)}(\psi^\lambda+\sum_{j=p+1}^k\kappa_jG(x^\lambda_j,\cdot)),\]
then we can define
\begin{equation}\label{314}
\begin{split}
c^\lambda_i&:=-\sup_{\{\omega^\lambda_i<0\}\cap B_\delta(\bar{x}_i)}(\psi^\lambda+\sum_{j=p+1}^k\kappa_jG(x^\lambda_j,\cdot))\\
&=-\inf_{\{\omega^\lambda_i>-\lambda\}\cap B_\delta(\bar{x}_i)}(\psi^\lambda+\sum_{j=p+1}^k\kappa_jG(x^\lambda_j,\cdot)).
\end{split}
\end{equation}
Similarly we have
\[\omega^\lambda_i=-\lambda I_{\{\psi^\lambda+\sum_{j=p+1}^k\kappa_jG(x_j^\lambda,\cdot)<-c^\lambda_i\}\cap B_\delta(\bar{x}_i)}.\]

To sum up, $\omega^\lambda_i$ has the form
\[\omega^\lambda_i=sgn(\kappa_i)\lambda I_{\{sgn(\kappa_i)(\psi^\lambda+\sum_{j=p+1}^k\kappa_jG(x_j^\lambda,\cdot))>c^\lambda_i\}\cap B_\delta(\bar{x}_i)},\]
which is the desired result.
\end{proof}

For the $c^\lambda_i$ given in Lemma \ref{3-007}, we have

\begin{lemma}\label{3-08}
$c^\lambda_i>-\frac{|\kappa_i|}{2\pi}\ln\delta-C$, where $C>0$ is independent of $\lambda$ and $\delta$.
\end{lemma}
\begin{proof}
We only prove the case $\kappa_i>0$. The case $\kappa_i<0$ can be proved similarly.

By the definition of $c^\lambda_i$(see \eqref{38}),
\begin{equation}\label{3-91}
\begin{split}
c^\lambda_i&=\inf_{\{\omega^\lambda_i>0\}\cap B_\delta(\bar{x}_i)}(\psi^\lambda+\sum_{j=p+1}^k\kappa_jG(x^\lambda_j,\cdot))\\
&\geq \inf_{B_\delta(\bar{x}_i)}(\psi^\lambda+\sum_{j=p+1}^k\kappa_jG(x^\lambda_j,\cdot))\quad\text{(by the maximum principle)}\\
&\geq \inf_{\partial B_\delta(\bar{x}_i)}(\psi^\lambda+\sum_{j=p+1}^k\kappa_jG(x^\lambda_j,\cdot))\\
&\geq \inf_{\partial B_\delta(\bar{x}_i)}\psi^\lambda-C
\end{split}
\end{equation}
for some $C>0$ not depending on $\lambda$ and $\delta$. But for any $x\in \partial B_\delta(\bar{x}_i)$,
\begin{equation}\label{3-92}
\begin{split}
\psi^\lambda(x)&=\int_DG(x,y)\omega^\lambda(y)dy\\
&=-\frac{1}{2\pi}\int_D\ln|x-y|\omega^\lambda_i(y)dy-\int_Dh(x,y)\omega^\lambda_i(y)dy\\
&+\int_DG(x,y)\sum_{j=1,j\neq i}^p\omega^\lambda_j(y)dy\\
&\geq -\frac{\kappa_i}{2\pi}\ln|2\delta|-C,
\end{split}
\end{equation}
where $C$ is a positive number not depending on $\lambda,\delta$. Here we use the fact that $dist(B_{\delta_0}(\bar{x}_i),B_{\delta_0}(\bar{x}_j))>0$ for any $i\neq j, 1\leq i,j\leq p$, and $dist(B_{\delta_0}(\bar{x}_i),\partial D)>0, 1\leq i\leq p$.

Combining \eqref{3-91} with \eqref{3-92} we get the desired result.
\end{proof}

\begin{lemma}\label{3-09}
For $\delta$ sufficiently small(not depending on $\lambda$) the following assertion holds true
 \begin{equation}\label{3-41}
 sgn(\kappa_i)(\psi^\lambda+\sum_{j=p+1}^k\kappa_jG(x^\lambda_j,\cdot))-c^\lambda_i\leq 0\quad\text{on $\partial B_{\delta_0}(\bar{x}_i)$}
 \end{equation}
  for each $1\leq i\leq p.$
\end{lemma}
\begin{proof}
Let $1\leq i\leq p$ be fixed. By Lemma \ref{3-08} it suffices to show that
\[\sup_{x\in \partial B_{\delta_0}(\bar{x}_i)}|\psi^\lambda(x)+\sum_{j=p+1}^k\kappa_jG(x^\lambda_j,x)|\leq C\]
for some $C>0$ not depending on $\lambda$ and $\delta.$  In fact, since for any $j\neq i, j=p+1,\cdot\cdot\cdot,k$, $dist(B_{\delta_0}(\bar{x}_i),B_{\delta_0}(\bar{x}_j))>0$, we have
\[\sup_{x\in \partial B_{\delta_0}}|\sum_{j=p+1}^k\kappa_jG(x^\lambda_j,x)|\leq C.\]
It remains to show that
\[\sup_{x\in \partial B_{\delta_0}(\bar{x}_i)}|\psi^\lambda(x)|\leq C.\]
In fact, for any $x\in \partial B_{\delta_0}(\bar{x}_i)$, by simple calculation
\begin{equation}
\begin{split}
|\psi^\lambda(x)|&\leq |\int_D-\frac{1}{2\pi}\ln|x-y|\omega^\lambda_i(y)dy|+|\int_Dh(x,y)\omega^\lambda_i(y)dy|\\
&+|\int_DG(x,y)\sum_{j\neq i,j=1}^p\omega^\lambda_j(y)dy|\\
&\leq-\frac{|\kappa_i|}{2\pi}\ln(\delta_0-\delta)+C \quad \text{(recall that $\delta<\frac{\delta_0}{2}$)}\\
&\leq-\frac{|\kappa_i|}{2\pi}\ln(\frac{\delta_0}{2})+C,
\end{split}
\end{equation}
which completes the proof.

\end{proof}

From now on we choose $\delta$ to be fixed such that $\eqref{3-41}$ holds true.

Now we turn to analyze the asymptotic behavior of the maximizer $(\omega^\lambda,x^\lambda_{p+1},\cdot\cdot\cdot,x^\lambda_k)$ as $\lambda\rightarrow+\infty$. More specifically, we will show that the support of $\omega^\lambda_i$ shrinks to $\bar{x}_i$ for $i=1,\cdot\cdot\cdot,p$ and $x^\lambda_j\rightarrow\bar{x}_j$ for $j=p+1,\cdot\cdot\cdot,k.$ To achieve this, we estimate the energy of the background vorticity first.

In the sequel we will use $C$ to denote various positive numbers not depending on $\lambda$.

\begin{lemma}\label{3-71}
$E(\omega^\lambda)\geq-\frac{1}{4\pi}\sum_{i=1}^p\kappa_i^2\ln\varepsilon_i-C$, where $\varepsilon_i$ is a positive number such that $\lambda\pi\varepsilon^2_i=|\kappa_i|.$
\end{lemma}
\begin{proof}
Choose test function $\omega=\sum_{i=1}^p\omega_i$, where $\omega_i=sgn(\kappa_i)\lambda I_{B_{\varepsilon_i}(\bar{x}_i)}.$ It is obvious that $\omega\in N_p^\lambda$, so by \eqref{35}
\[E(\omega)+\sum_{j=p+1}^k\kappa_j\psi(x_j^\lambda)\leq E(\omega^\lambda)+\sum_{j=p+1}^k\kappa_j\psi^\lambda(x_j^\lambda),\]
where $\psi=(-\Delta)^{-1}\omega$. It is easy to check that
\[|\sum_{j=p+1}^k\kappa_j\psi(x_j^\lambda)|\leq C,\quad |\sum_{j=p+1}^k\kappa_j\psi^\lambda(x_j^\lambda)|\leq C,\]
so
\begin{equation}\label{3-20}
E(\omega^\lambda)\geq E(\omega)-C.
\end{equation}
On the other hand,
\begin{equation}\label{3-21}
\begin{split}
E(\omega)&=\frac{1}{2}\int_D\int_DG(x,y)\omega(x)\omega(y)dxdy\\
&=\sum_{i=1}^pE(\omega_i)+\sum_{1\leq i<j\leq p}\int_D\int_DG(x,y)\omega_i(x)\omega_j(y)dxdy\\
&\geq \sum_{i=1}^pE(\omega_i)-C.
\end{split}
\end{equation}
Each $E(\omega_i)$ can be calculated directly:
\begin{equation}\label{3-22}
\begin{split}
E(\omega_i)&=\frac{1}{2}\int_D\int_DG(x,y)\omega_i(x)\omega_i(y)dxdy\\
&=\frac{1}{2}\int_{B_{\varepsilon_i}(\bar{x}_i)}\int_{B_{\varepsilon_i}(\bar{x}_i)}-\frac{1}{2\pi}\ln|x-y|\omega_i(x)\omega_i(y)dxdy\\
&-\frac{1}{2}\int_{B_{\varepsilon_i}(\bar{x}_i)}\int_{B_{\varepsilon_i}(\bar{x}_i)}h(x,y)\omega_i(x)\omega_i(y)dxdy\\
&\geq-\frac{\kappa_i^2}{4\pi}\ln(2\varepsilon_i)-\sum_{i=1}^p\kappa_i^2H(\bar{x}_i)+o(1),
\end{split}
\end{equation}
where $o(1)\rightarrow0 $ as $\lambda\rightarrow+\infty$. Now \eqref{3-20},\eqref{3-21} and \eqref{3-22} together give the desired result.
\end{proof}

Now we define
\begin{equation}
T^\lambda:=\sum_{i=1}^p\frac{sgn(\kappa_i)}{2}\int_D\omega^\lambda_i(x)[sgn(\kappa_i)(\psi^\lambda(x)+\sum_{j=p+1}^k\kappa_j G(x_j^\lambda,x))-c^\lambda_i]dx,
\end{equation}
which represents the total kinetic energy of the fluid on the vorticity set $\cup_{i=1}^pA^\lambda_i$.

\begin{lemma}\label{3-72}
$T^\lambda\leq C.$
\end{lemma}

\begin{proof}
For simplicity we write \[\zeta^\lambda_i:=sgn(\kappa_i)(\psi^\lambda+\sum_{j=p+1}^k\kappa_j G(x_j^\lambda,\cdot))-c^\lambda_i.\]
It suffices to prove that $sgn(\kappa_i)\int_D\omega^\lambda_i\zeta^\lambda_idx\leq C$ for each $1\leq i\leq p.$

On the one hand, by H\"{o}lder's inequality and Sobolev embedding $W^{1,1}(B_\delta(\bar{x}_i))\hookrightarrow L^2(B_\delta(\bar{x}_i))$, we have
\begin{equation}
\begin{split}
&sgn(\kappa_i)\int_D\omega^\lambda_i\zeta^\lambda_idx
=\lambda\int_{A^\lambda_i}\zeta^\lambda_idx\\
\leq&\lambda|A^\lambda_i|^{\frac{1}{2}}(\int_{A^\lambda_i}|\zeta^\lambda_i|^2dx)^{\frac{1}{2}}\quad(\zeta^\lambda_i\geq0\,\, \text{on}\,\, A^\lambda_i)\\
\leq&\lambda|A^\lambda_i|^{\frac{1}{2}}(\int_{B_\delta(\bar{x}_i)}|(\zeta^\lambda_i)^+|^2dx)^{\frac{1}{2}}\\
\leq&C\lambda|A^\lambda_i|^{\frac{1}{2}}(\int_{B_\delta(\bar{x}_i)}(\zeta^\lambda_i)^+dx+\int_{B_\delta(\bar{x}_i)}|\nabla(\zeta^\lambda_i)^+|dx)\\
\leq&C\lambda|A^\lambda_i|^{\frac{1}{2}}(\int_{A^\lambda_i}\zeta^\lambda_idx+\int_{A^\lambda_i}|\nabla\zeta^\lambda_i|dx)\\
\leq&C\lambda|A^\lambda_i|^{\frac{1}{2}}|A^\lambda_i|^{\frac{1}{2}}(\int_{A^\lambda_i}|\nabla \zeta^\lambda_i|^2dx)^{\frac{1}{2}}+C|A^\lambda_i|^\frac{1}{2}sgn(\kappa_i)\int_D\omega^\lambda_i\zeta^\lambda_idx\\
\leq&C(\int_{A^\lambda_i}|\nabla \zeta^\lambda_i|^2dx)^{\frac{1}{2}}+o(1)sgn(\kappa_i)\int_D\omega^\lambda_i\zeta^\lambda_idx,
\end{split}
\end{equation}
which implies
\begin{equation}\label{3-100}
sgn(\kappa_i)\int_D\omega^\lambda_i\zeta^\lambda_idx\leq C(\int_{A^\lambda_i}|\nabla \zeta^\lambda_i|^2dx)^{\frac{1}{2}}.
\end{equation}

On the other hand, since $\zeta^\lambda_i\leq 0$ on $\partial B_{\delta_0}(\bar{x}_i)$(see \eqref{3-41}), integration by parts gives
\begin{equation}\label{3-101}
\begin{split}
sgn(\kappa_i)\int_D\omega^\lambda_i\zeta^\lambda_idx&=sgn(\kappa_i)\int_{B_{\delta_0}(\bar{x}_i)}\omega^\lambda_i(\zeta^\lambda_i)^+dx\\
&=\int_{B_{\delta_0}(\bar{x}_i)}|\nabla(\zeta^\lambda_i)^+|^2dx\\
&\geq \int_{A^\lambda_i}|\nabla\zeta^\lambda_i|^2dx.
\end{split}
\end{equation}

Combining \eqref{3-100} with \eqref{3-101} we complete the proof.
\end{proof}

\begin{lemma}\label{3-73}
$\sum_{i=1}^pc_i^\lambda|\kappa_i|\geq -\frac{1}{2\pi}\sum_{i=1}^p\kappa_i^2\ln\varepsilon_i-C.$
\end{lemma}
\begin{proof}
By the definition of $T^\lambda$, the following identity holds true
\[T^\lambda=E(\omega^\lambda)+\frac{1}{2}\sum_{i=1}^p\sum_{j=p+1}^k\int_D\omega^\lambda_i(x)\kappa_jG(x^\lambda_j,x)dx-\frac{1}{2}\sum_{i=1}^pc^\lambda_i|\kappa_i|.\]
Since $dist(B_{\delta_0}(\bar{x}_i),B_{\delta_0}(\bar{x}_j))>0$ for any $i=1,\cdot\cdot\cdot,p$ and $j=p+1,\cdot\cdot\cdot, k$, it is easy to see that
\[|\frac{1}{2}\sum_{i=1}^p\sum_{j=p+1}^k\int_D\omega^\lambda_i(x)\kappa_jG(x^\lambda_j,x)dx|\leq C,\]
then the desired result follows from Lemma \ref{3-71} and Lemma \ref{3-72}.
\end{proof}

Now we are ready to estimate the size of $supp(\omega^\lambda_i)$.

\begin{lemma}

There exists $R_0>0$ such that $diam(supp(\omega^\lambda_i))\leq R_0\varepsilon_i$, $i=1,\cdot\cdot\cdot,p.$
\end{lemma}
\begin{proof}
For any $x_i\in supp(\omega^\lambda_i)$, $i=1,\cdot\cdot\cdot,p$, we have
\[sgn(\kappa_i)(\psi^\lambda(x_i)+\sum_{j=p+1}^k\kappa_jG(x^\lambda_j,x_i))\geq c^\lambda_i.\]
It is easy to see that
\[|\sum_{j=p+1}^k\kappa_jG(x^\lambda_j,x_i)|\leq C,\]
so we have
\[sgn(\kappa_i)\psi^\lambda(x_i)\geq c^\lambda_i-C,\]
which gives
\begin{equation}\label{3-073}
\int_D-\frac{1}{2\pi}\ln|x_i-y||\omega^\lambda_i(y)|dy\geq c^\lambda_i-C.
\end{equation}
Combining \eqref{3-073} with Lemma \ref{3-73} we obtain
\[\sum_{i=1}^p|\kappa_i|\int_D-\frac{1}{2\pi}\ln|x_i-y||\omega^\lambda_i(y)|dy\geq-\frac{1}{2\pi}\sum_{i=1}^p\kappa_i^2\ln\varepsilon_i-C,\]
or equivalently
\[\sum_{i=1}^p\frac{|\kappa_i|}{2\pi}\int_D\ln\frac{\varepsilon_i}{|x_i-y|}|\omega^\lambda_i(y)|dy\geq-C.\]
So for any $R>1$ to be determined, we have
\begin{equation}\label{3-77}
\begin{split}
&\sum_{i=1}^p\frac{|\kappa_i|}{2\pi}\int_{D\setminus B_{R\varepsilon_i}(x_i)}\ln\frac{\varepsilon_i}{|x_i-y|}|\omega^\lambda_i(y)|dy\\
+&\sum_{i=1}^p\frac{|\kappa_i|}{2\pi}\int_{B_{R\varepsilon_i}(x_i)}\ln\frac{\varepsilon_i}{|x_i-y|}|\omega^\lambda_i(y)|dy\geq-C.
\end{split}
\end{equation}
The second integral in \eqref{3-77} is bounded(in fact, it can be calculated explicitly), that is,
\[|\sum_{i=1}^p\frac{|\kappa_i|}{2\pi}\int_{B_{R\varepsilon_i}(x_i)}\ln\frac{\varepsilon_i}{|x_i-y|}|\omega^\lambda_i(y)|dy|\leq C,\]
so we get
\begin{equation}
\sum_{i=1}^p\frac{|\kappa_i|}{2\pi}\int_{D\setminus B_{R\varepsilon_i}(x_i)}\ln\frac{\varepsilon_i}{|x_i-y|}|\omega^\lambda_i(y)|dy\geq-C,
\end{equation}
which implies
\begin{equation}
\sum_{i=1}^p\frac{|\kappa_i|}{2\pi}\int_{D\setminus B_{R\varepsilon_i}(x_i)}|\omega^\lambda_i(y)|dy\leq \frac{C}{\ln R},
\end{equation}
so for each $1\leq i\leq p$,
\begin{equation}
\frac{|\kappa_i|}{2\pi}\int_{D\setminus B_{R\varepsilon_i}(x_i)}|\omega^\lambda_i(y)|dy\leq \frac{C}{\ln R}.
\end{equation}
Since $\int_D|\omega^\lambda_i(y)|dy=|\kappa_i|$, by choosing $R$ large enough we get
\begin{equation}\label{3-80}
\int_{B_{R\varepsilon_i}(x_i)}|\omega^\lambda_i(y)|dy> \frac{|\kappa_i|}{2},\quad i=1,\cdot\cdot\cdot,p.
\end{equation}

Now we claim that \[diam(supp(\omega^\lambda_i))\leq 2R\varepsilon_i.\]
In fact, supposing that $diam(supp(\omega^\lambda_i))> 2R\varepsilon_i$, we can choose $x_i,y_i\in supp(\omega^\lambda_i)$ such that $|x_i-y_i|>2R\varepsilon_i$, then by \eqref{3-80}(recall that in \eqref{3-80} $x_i\in supp(\omega^\lambda_i)$ is arbitrary)
\[\int_D|\omega^\lambda_i(y)|dy\geq\int_{B_{R\varepsilon_i}(x_i)}|\omega^\lambda_i(y)|dy+\int_{B_{R\varepsilon_i}(y_i)}|\omega^\lambda_i(y)|dy>|\kappa_i|,\]
which is a contradiction.

By choosing $R_0=2R$ we complete the proof.

\end{proof}

By now we have constructed $\omega^\lambda_i\in N_p^\lambda, i=1,\cdot\cdot\cdot,p$, and $x^\lambda_j\in \overline{B_\delta(\bar{x}_j)}, j=p+1\cdot\cdot\cdot,k$. Moreover, it is proved that the diameter of $supp(\omega^\lambda_i)$ vanishes as $\lambda\rightarrow+\infty.$ To analyze their limiting positions, we define the center of $\omega^\lambda_i$ as follows
\[z^\lambda_i:=\frac{1}{\kappa_i}\int_Dx\omega^\lambda_idx,\quad i=1,\cdot\cdot\cdot,p.\]
It is easy to see that $z^\lambda_i\in \overline{B_\delta(\bar{x}_i)}$.
\begin{lemma}\label{3-99}
$z^\lambda_i\rightarrow\bar{x}_i$ for $1\leq i\leq p$ and $x^\lambda_j\rightarrow \bar{x}_j$ for $p+1\leq j\leq k$ as $\lambda\rightarrow+\infty.$
\end{lemma}
\begin{proof}
Up to a subsequence we assume that $z^\lambda_i\rightarrow z_i\in \overline{B_\delta(\bar{x}_i)}, 1\leq i\leq p$, and
$x^\lambda_j\rightarrow z_j\in \overline{B_\delta(\bar{x}_j)}, p+1\leq j\leq k$. It suffices to show that $(z_1,\cdot\cdot\cdot,z_k)=(\bar{x}_1,\cdot\cdot\cdot,\bar{x}_k)$.

Define a test function $\omega=\sum_{i=1}^p\omega_i,$ where $\omega_i=sgn(\kappa_i)\lambda I_{B_{\varepsilon_i}(\bar{x}_i)}$. It is easy to see that $\omega\in N_p^\lambda$, so we have
\[\mathcal{E}(\omega,\bar{x}_{p+1},\cdot\cdot\cdot,\bar{x}_k)\leq \mathcal{E}(\omega^\lambda,{x}^\lambda_{p+1},\cdot\cdot\cdot,{x}^\lambda_k),\]
that is,
\begin{equation}
\begin{split}
&E(\omega)+\sum_{j=p+1}^k\kappa_j\psi(\bar{x}_j)+\frac{1}{2}\sum_{\begin{subarray}\quad \quad \,\,i\neq j\\p+1\leq i,j\leq k\end{subarray}}\kappa_i\kappa_jG(\bar{x}_i,\bar{x}_j)-\sum_{j=p+1}^k\kappa_j^2H(\bar{x}_j)\\
\leq & E(\omega^\lambda)+\sum_{j=p+1}^k\kappa_j\psi^\lambda({x}^\lambda_j)+\frac{1}{2}\sum_{\begin{subarray}\quad \quad \,\,i\neq j\\p+1\leq i,j\leq k\end{subarray}}\kappa_i\kappa_jG({x}^\lambda_i,{x}^\lambda_j)-\sum_{j=p+1}^k\kappa_j^2H({x}^\lambda_j).
\end{split}
\end{equation}
where $\psi=(-\Delta)^{-1}\omega$ and $\psi^\lambda=(-\Delta)^{-1}\omega^\lambda$. It is easy to check that
\begin{equation}
\begin{split}
E(\omega)&=\frac{1}{2}\sum_{i=1}^p\int_D\int_D-\frac{1}{2\pi}\ln|x-y|\omega_i(x)\omega_i(y)dxdy-\sum_{i=1}^p\kappa_i^2H(\bar{x}_i)\\
&+\frac{1}{2}\sum_{\begin{subarray}\quad \,\,\,\,\,\,i\neq j\\1\leq i,j\leq p\end{subarray}}\kappa_i\kappa_jG(\bar{x}_i,\bar{x}_j)+o(1),
\end{split}
\end{equation}
and
\begin{equation}
\sum_{j=p+1}^k\kappa_j\psi(\bar{x}_j)=\sum_{i=1}^p\sum_{j=p+1}^k\kappa_i\kappa_jG(\bar{x}_i,\bar{x}_j)+o(1),
\end{equation}
henceforth we use $o(1)$ to denote quantities that go to 0 as $\lambda\rightarrow+\infty.$ Similarly
\begin{equation}
\begin{split}
E(\omega^\lambda)&=\frac{1}{2}\sum_{i=1}^p\int_D\int_D-\frac{1}{2\pi}\ln|x-y|\omega^\lambda_i(x)\omega^\lambda_i(y)dxdy-\sum_{i=1}^p\kappa_i^2H(z_i)\\
&+\frac{1}{2}\sum_{\begin{subarray}\quad \,\,\,\,\,\, i\neq j\\1\leq i,j\leq p\end{subarray}}\kappa_i\kappa_jG(z_i,z_j)+o(1),
\end{split}
\end{equation}
and
\begin{equation}
\sum_{j=p+1}^k\kappa_j\psi^\lambda({x}^\lambda_j)=\sum_{i=1}^p\sum_{j=p+1}^k\kappa_i\kappa_jG(z_i,z_j)+o(1).
\end{equation}
Hence from all the above we obtain
\begin{equation}
\begin{split}
&\frac{1}{2}\sum_{i=1}^p\int_D\int_D-\frac{1}{2\pi}\ln|x-y|\omega_i(x)\omega_i(y)dxdy-\sum_{i=1}^p\kappa_i^2H(\bar{x}_i)\\
&+\frac{1}{2}\sum_{\begin{subarray}\quad \,\,\,\,\,\,i\neq j\\1\leq i,j\leq p\end{subarray}}\kappa_i\kappa_jG(\bar{x}_i,\bar{x}_j)
+\sum_{i=1}^p\sum_{j=p+1}^k\kappa_i\kappa_jG(\bar{x}_i,\bar{x}_j)\\
&+\frac{1}{2}\sum_{\begin{subarray}\quad \quad \,\,i\neq j\\p+1\leq i,j\leq k\end{subarray}}\kappa_i\kappa_jG(\bar{x}_i,\bar{x}_j)-\sum_{j=p+1}^k\kappa_j^2H(\bar{x}_j)\\
\leq & \frac{1}{2}\sum_{i=1}^p\int_D\int_D-\frac{1}{2\pi}\ln|x-y|\omega^\lambda_i(x)\omega^\lambda_i(y)dxdy-\sum_{i=1}^p\kappa_i^2H(z_i)\\
&+\frac{1}{2}\sum_{\begin{subarray}\quad \,\,\,\,\,\, i\neq j\\1\leq i,j\leq p\end{subarray}}\kappa_i\kappa_jG(z_i,z_j)
+\sum_{i=1}^p\sum_{j=p+1}^k\kappa_i\kappa_jG(z_i,z_j)\\
&+\frac{1}{2}\sum_{\begin{subarray}\quad \quad \,\,i\neq j\\p+1\leq i,j\leq k\end{subarray}}\kappa_i\kappa_jG(z_i,z_j)-\sum_{j=p+1}^k\kappa_j^2H(z_j)+o(1).
\end{split}
\end{equation}
On the other hand, by Riesz's rearrangement inequality
\begin{equation}
\begin{split}
&\int_D\int_D-\frac{1}{2\pi}\ln|x-y|\omega^\lambda_i(x)\omega^\lambda_i(y)dxdy\\
\leq& \int_D\int_D-\frac{1}{2\pi}\ln|x-y|\omega_i(x)\omega_i(y)dxdy,\;\;i=1,\cdot\cdot\cdot,p,
\end{split}
\end{equation}
so we get
\begin{equation}
\begin{split}
&\frac{1}{2}\sum_{\begin{subarray}\quad\,\,\,\,\,\, i\neq j\\1\leq i,j\leq k\end{subarray}}\kappa_i\kappa_jG(\bar{x}_i,\bar{x}_j)-\sum_{j=1}^k\kappa_j^2H(\bar{x}_j)\\
\leq &\frac{1}{2}\sum_{\begin{subarray}\quad \quad \,\,i\neq j\\1\leq i,j\leq k\end{subarray}}\kappa_i\kappa_jG(z_i,z_j)-\sum_{j=1}^k\kappa_j^2H(z_j),
\end{split}
\end{equation}
or equivalently
\begin{equation}
\mathcal{K}_k(z_1,\cdot\cdot\cdot,z_k)\leq \mathcal{K}_k(\bar{x}_1,\cdot\cdot\cdot,\bar{x}_k).
\end{equation}
Since $(\bar{x}_1,\cdot\cdot\cdot,\bar{x}_k)$ is the unique minimum point for $\mathcal{K}_k$ on $\overline{B_\delta(\bar{x}_1)}\times\cdot\cdot\cdot\times\overline{B_\delta(\bar{x}_k)}$, we obtain
\[(z_1,\cdot\cdot\cdot,z_k)=(\bar{x}_1,\cdot\cdot\cdot,\bar{x}_k),\]
which completes the proof.
\end{proof}
\begin{remark}
It is easy to see that $x^\lambda_j\in B_\delta(\bar{x}_j)$ for $ j=p+1,\cdot\cdot\cdot,k,$ and $dist(supp(\omega_i^\lambda),\partial B_\delta(\bar{x}_i))>0$ for $i=1,\cdot\cdot\cdot,p$, provided that $\lambda$ is sufficiently large.
\end{remark}
\section{Proof of Theorem \ref{230}}
Having made all preparation in the preceding sections, in this section we finish the proof of Theorem \ref{230}.

\begin{proof}[Proof of Theorem \ref{230}]
We need only to prove that $(\omega^\lambda,x^\lambda_{p+1},\cdot\cdot\cdot,x^\lambda_k)$ satisfies \eqref{29} if $\lambda$ is sufficiently large.

Firstly by Lemma \ref{3-99}, $x^\lambda_j\in B_\delta(\bar{x}_j)$ for $ j=p+1,\cdot\cdot\cdot,k,$ then by \eqref{36} we have

\[\nabla(\psi^\lambda(x)+\sum_{j=p+1,j\neq i}^k\kappa_jG({x}^\lambda_j,x)-\kappa_i H(x))\big|_{x={x}^\lambda_i}=0,\,\,j=p+1,\cdot\cdot\cdot,k.\]

Now for any given $\phi\in C^{\infty}_c(D)$, define a family of transformations $\Phi_t(x), t\in(-\infty,+\infty)$, from $D$ to $D$ by the following ordinary differential equation:
\begin{equation}\label{400}
\begin{cases}\frac{d\Phi_t(x)}{dt}=J\nabla\phi(\Phi_t(x)),\,\,\,t\in\mathbb R, \\
\Phi_0(x)=x.
\end{cases}
\end{equation}
Since $J\nabla\phi$ is a smooth vector field with compact support in $D$, $\eqref{400}$ is solvable for all $t\in \mathbb{R}$. It's also easy to see that $J\nabla\phi$ is divergence-free, so by Liouville theorem(see \cite{MPu}, Appendix 1.1) $\Phi_t(x)$ is area-preserving, or equivalently for any measurable set $A\subset D$
\begin{equation}
|\Phi_t(A)|=|A|.
\end{equation}
 Now define a family of test function
\begin{equation}
\omega^\lambda_t(x):=\omega^\lambda(\Phi_{-t}(x)).
\end{equation}
Since $\Phi_t$ is area-preserving and $dist(supp(\omega_i^\lambda),\partial B_\delta(\bar{x}_i))>0$ for each $1\leq i\leq p$, for $|t|$ is sufficiently small we have
$\omega^\lambda_t\in N_p^\lambda.$
Then by \eqref{35} we have
\[\frac{d}{dt}(E(\omega^\lambda_t)+\sum_{j=p+1}^k\kappa_j\psi^\lambda_t(x^\lambda_j))\big|_{t=0}=0,\]
where $\psi^\lambda_t=(-\Delta)^{-1}\omega^\lambda_t.$
It is easy to check that(see \cite{T} for example)
\[E(\omega^\lambda_t)=E(\omega^\lambda)+t\int_D\omega^\lambda(x)\partial(\psi^\lambda(x),\phi(x))dx+o(t),\]
and
\[\psi^\lambda_t(x^\lambda_j)=\psi^\lambda(x^\lambda_j)+t\int_D\omega^\lambda(x)\partial(G(x^\lambda_j,x),\phi(x))dx+o(t),\]
where $\frac{o(t)}{t}\rightarrow0$ as $t\rightarrow0$, therefore we get
\[\int_D\omega^\lambda(x)\partial(\psi^\lambda(x)+\sum_{j=p+1}^k\kappa_jG(x^\lambda_j,x),\phi(x))dx=0,\]
which completes the proof of Theorem \ref{230}.

\end{proof}

\smallskip

\renewcommand\refname{References}
\renewenvironment{thebibliography}[1]{%
\section*{\refname}
\list{{\arabic{enumi}}}{\def\makelabel##1{\hss{##1}}\topsep=0mm
\parsep=0mm
\partopsep=0mm\itemsep=0mm
\labelsep=1ex\itemindent=0mm
\settowidth\labelwidth{\small[#1]}%
\leftmargin\labelwidth \advance\leftmargin\labelsep
\advance\leftmargin -\itemindent
\usecounter{enumi}}\small
\def\newblock{\ }
\sloppy\clubpenalty4000\widowpenalty4000
\sfcode`\.=1000\relax}{\endlist}
\bibliographystyle{model1b-num-names}

\end{document}